\documentclass[11pt]{amsart}
\pdfoutput=1

\pdfpageattr{/Group <</S /Transparency /I true /CS /DeviceRGB>>}

\usepackage[T1]{fontenc}
\usepackage[english]{babel}
\usepackage{geometry}
\usepackage{amsfonts,amsmath,amssymb,amsthm,mathrsfs}
\usepackage{tikz}
\usetikzlibrary{calc}
\usepackage{caption}
\usepackage{subcaption}
\usepackage{xspace}
\usepackage{booktabs,tabularx}
\usepackage{dcolumn}
\usepackage{marvosym}
\usepackage{microtype}
\usepackage{enumerate}
\usepackage{mathtools}
\usepackage{booktabs}
\usepackage{url}
\usepackage{hyperref}

\microtypesetup{tracking,kerning,spacing}
\microtypecontext{spacing=nonfrench}

 \hypersetup{colorlinks,linkcolor=blue,citecolor=blue}

\newcommand\KK{{\mathbb K}}

\newcommand\FF{{\mathbb F}}
\newcommand\QQ{{\mathbb Q}}

\newcommand\cP{\mathcal{P}}

\newcommand\cC{{\mathcal C}}

\newcommand\cB{{\mathcal B}}

\newcommand\smallSetOf[2]{\{{#1}\,|\,{#2}\}}

\theoremstyle{plain}
    \newtheorem{theorem}{Theorem}
    
    \newtheorem{lemma}[theorem]{Lemma}
    \newtheorem{proposition}[theorem]{Proposition}
\theoremstyle{definition}
    
    \newtheorem{example}[theorem]{Example}
    \newtheorem{definition}[theorem]{Definition}

    \newtheorem{question}{Question}

\title{The correlation constant of a field}

\author{Benjamin Schr\"oter}

\address[Benjamin Schr\"oter]{
  Institut f{\"u}r Mathematik,
  TU Berlin,
  Str.\ des 17. Juni 136, 10623 Berlin, Germany
}
\email{schroeter@math.tu-berlin.de}


\begin{document}

\begin{abstract}
	We study the correlation of edges, vectors or elements to be in a randomly chosen spanning tree or a basis, respectively.
	Here we follow the guideline of Huh and Wang and introduce as a measure an invariant that is called the correlation constant of a graph, vector configuration, matroid or field. It follows from one of their results that these correlation constants are numbers between $0$ and $2$.
	Here, we show that the correlation constant of every field is at least $\frac{8}{7}$.
	In our proof we explicitly construct vector configurations and matroids with positively correlated elements.
\end{abstract}
\maketitle

\section{Introduction}
\noindent
This article deals with a basic question which appears in both the theory of graphs and finite geometries and ask how strongly independence of edges in a graph or vectors in a configuration is correlated.

First we consider graphs and their edges. Let $G$ be a finite connected graph, $i$, $j$ two edges of $G$ and consider a uniform distribution on the spanning trees of $G$. We denote by $\Pr(i\in T)$ the probability that the edge $i$ is in a randomly chosen tree $T$. 
\begin{question}\label{Q:graphs}
What is the correlation between the probabilities $\Pr(i\in T)$ and $\Pr(j\in T)$ for distinguishable edges $i$ and $j$?
\end{question}
From Kirchhoff's law in an electrical network Brooks, Smith, Stone and Tutte derive an equation \cite[Equation (2.34)]{BrooksSmithStoneTutte:1940} which implies that the above events are negatively correlated, i.e.,
the covariance $\Pr(i,j\in T) - \Pr(i\in T)\cdot \Pr(j\in T)$ is negative. This plays a central role in Tutte's characterization of graphs with a constant number of spanning trees through any two edges; see \cite{Tutte:1974}.

Now we take a look at vectors. 
Given a field $\KK$, let $\cP$ be a vector configuration in a $\KK$-vector space with a uniform distribution on the basis formed by vectors of $\cP$.
The central number of this article is the \emph{correlation constant} $\beta(\cP) = \max_{v,w\in\cP}\frac{\Pr(v,w\in B)}{\Pr(v\in B)\cdot\Pr(w\in B)}$ of the configuration $\cP$, where $B$ is a randomly chosen basis.
Huh and Wang \cite{HuhWang:2017} asked for the following.
\begin{question}\label{Q:HuhWang}
	How large can the correlation constant (for a given field) be?
\end{question}
The \emph{correlation constant} $\beta_\KK$ of a field is the supremum of all correlation constants taken over all vector configurations.
The aim of this article is to give an explicit lower bound on this correlation constant.

The common language of graphs and finite vector configurations is matroid theory. The monographs of Oxley \cite{Oxley:2011} and White \cite{White:1986} serve as the foundation for this article.
A \emph{matroid} $M$ is a non-empty collection $\cB$ of subsets of a finite set $E$ with the property that for every pair $B,B'\in\cB$ and any element $e\in B\setminus B'$ an element $e'\in B'\setminus B$ exists such that $B\setminus\{e\}\cup\{e'\}\in\cB$.
The set $E$ is called the \emph{ground set}, the sets in $\cB$ are the \emph{bases} of the matroid $M$ and a \emph{loop} is an element that does not occur in any basis.
In the following we will assume that the ground set $E$ is $[n]=\{1,2,\ldots,n\}$.
The questions above lead to the following definition.

\begin{definition}\label{def:beta}
Let both $i$ and $j$ be elements of $M$ that are not loops. Then we define
\[
	\beta(M;i,j) \coloneqq \frac{b \cdot b_{ij}}{b_i\cdot b_j} \enspace .
\]
	Where $b$, $b_i$, $b_j$ and $b_{ij}$ are the numbers of bases of $M$, the number of bases containing $i$, $j$ or $i,j$, respectively.
Assume that $M$ has at least two non-loops, then the \emph{correlation constant} of $M$ is the number $\beta(M)\coloneqq \max_{i,j}\beta(M;i,j)$, where the maximum ranges over all non loops $i$, $j$ of $M$.
\end{definition}

Question~\ref{Q:graphs} asks about the correlation constant of graphical matroids, while Question~\ref{Q:HuhWang} is about the correlation constant of $\KK$-representable matroids.

Seymour and Welsh \cite[Conjecture 4]{SeymourWelsh:1975} conjectured that the correlation constant of any matroid is bounded by one, i.e., the elements of a matroid are negatively correlated.
About fifty years ago Rota conjectured that the coefficients of the characteristic polynomial of a matroid are log-concave, the coefficients $a_k$ satisfy $a^2_k \geq a_{k−1}\cdot a_{k+1}$. 
Seymour and Welsh  \cite{SeymourWelsh:1975} claimed that their conjecture implies log-concavity of the coefficients.
However, they found a rank four binary matroid on eight elements with correlation constant $\frac{48\cdot 12}{20 \cdot28}=\frac{36}{35}\approx 1.02$; cf.~Example~\ref{ex:SeymourWelsh} below. This counterexample has been originally published as a note added in proof \cite{SeymourWelsh:1975}.

Recently, Adiprasito, Huh and Katz \cite{AdiprasitoHuhKatz:2015} proved Rota's conjecture as a key they introduced a version of the Hodge-Riemann bilinear relations for the Chow ring of a matroid. Surveys to their remarkable techniques are \cite{AdiprasitoHuhKatz:2017} and \cite{Baker:2018}.
Huh and Wang \cite{HuhWang:2017} developed a variant of their combinatorial approach using a Hodge-Riemann form over a M\"obius algebra which is generated by the variables associated with the elements in the ground set of a matroid.
A bilinear pairing between two of these variables is given by the numbers $b_{ij}$ from above if $i\neq j$ and $0$ otherwise. 
They claim in \cite{HuhWang:2017} that the Hodge-Riemann form of a simple matroid has exactly one positive eigenvalue and that a proof will be part of \cite{HuhWang}.
Cauchy's interlacing for symmetric matrices shows that this property is preserved when restricting to a well chosen $3$-dimensional subspace. The obtained form has a positive determinant, which is $2 (r-1)^2 \cdot b_i b_j b_{ij}- r(r-1)\cdot b b_{ij}^2$. Hence the correlation constant of every matroid is bounded by two.

As mentioned before, the aim of this article is the following new lower bound on the correlation constant $\beta_\KK$ of a field, i.e., the suprema of correlation constants for $\KK$-representable matroids.
\begin{theorem}
	The correlation constant $\beta_\KK$ of the field $\KK$ satisfies $\frac{8}{7}\leq \beta_\KK \leq 2$.
\end{theorem}
This statement is a central part of Theorem~\ref{thm:final}. As a tool we introduce another invariant $\alpha(M)$ of a matroid $M$; cf.~Definition~\ref{def:alpha}. In Theorem~\ref{thm:relation} we show that this $\alpha$-ratio is an upper bound for the correlation constant of a matroid with positive correlation.

We apply Theorem~\ref{thm:relation} to the large class of sparse paving matroids whose $\alpha$-ratio is bounded by one.
Hence we conclude that elements in these matroids are negatively correlated.
It is conjectured by Mayhew, Newman, Welsh and Whittle in \cite[Conjecture 1.6]{MayhewNewmanWelshWhittle:2011} that almost all matroids are sparse paving. This conjecture has its origin in a question of Welsh stated in \cite{Welsh:1971}.

We continue by constructing a sequence of matroids $M_k$ that we derive from a matroid $M$, such that the sequence of correlation constants $\beta(M_k)$ converges monotonically to $\alpha(M)$; see Lemma~\ref{lem:convergence}.
This completes our comparison of the $\alpha$-ratio and the correlation constant of a matroid.

In the last section, we present new examples of matroids with positive correlation and apply Theorem~\ref{thm:relation} and Lemma~\ref{lem:convergence} from the previous section.
In particular, we deduce Theorem~\ref{thm:final} and that the correlation constant of any field is at least $\frac{8}{7}\approx 1.14$.

\section{The $\alpha$-ratio and the correlation constant of a matroid}
\noindent
For our next steps we introduce further notation.
We denote by $b^{j}_{i}=b_i-b_{ij}$ the number of bases of $M$ containing $i$ and not $j$. Similar we define $b^{i}_{j}=b_j-b_{ij}$ and $b^{ij}=b-b_i-b^i_j$. Now we are able to give a definition of a second ratio. 
\begin{definition}\label{def:alpha}
	Let $i,j$ neither be loops, coloops nor parallel elements in $M$, then we define
	\[
		\alpha(M;i,j) \coloneqq \frac{  b^{ij}\cdot b_{ij}  }{  b_i^j \cdot b_j^i } \enspace .
	\]
	Let $M$ be a matroid that has at least a valid pair of elements, then the \emph{$\alpha$-ratio} $\alpha(M)$ of $M$ is the maximum $\max_{i,j} \alpha(M;i,j)$ over all valid pairs of elements $i,j$.
\end{definition}

Note that the numbers occurring in Definition~\ref{def:beta} and Definition~\ref{def:alpha} are the numbers of bases in deletions and contractions of the two elements $i$ and $j$.

As a first example for these definitions let us take a look at uniform matroids.
A matroid  whose collection of bases is formed by all $r$-sets of $[n]$ is called the \emph{uniform matroid} $U_{r,n}$. Clearly the number of its bases is given by the binomial coefficient $\tbinom{n}{r}$. The class of uniform matroids is minor closed, i.e., deletions and contractions are again uniform matroids.

\begin{example} 
	Let $1 < r < n$. We get for every pair $i, j$ in the uniform matroid $U_{r,n}$ 
	\[
		\alpha(U_{r,n},i,j)
		= \frac{(r-1)\cdot(n-r-1)}{r\cdot(n-r)}
		\ \text{ and } \
		\beta(U_{r,n},i,j) 
		= \frac{n \cdot (r-1)}{r \cdot(n-1)}
	\]
	and hence we have $0 \ < \ \alpha(U_{r,n}) \ < \ \beta(U_{r,n}) \ < \ 1 $.
\end{example}

We do not define an $\alpha$-ratio for the uniform matroids $U_{0,n}$, $U_{1,n}$ and $U_{n,n}$.
In general, each pair of elements in a matroid is either parallel, contains a loop or a coloop if and only if the matroid is a direct sum of the form $U_{1,n_1}\oplus U_{0,n_2}\oplus U_{n_3,n_3}$.

We start now with analyzing the relation between the $\alpha$-ratio and the correlation constant of a matroid.

\begin{proposition} Let $M = M_1\oplus M_2$ be a disconnected matroid and $i,j,\ell$ are neither loops, coloops nor pairwise parallel. If $i$ and $j$ are in the same connected component $M_1$, then we have
	\[
		\alpha(M;i,j) = \alpha(M_1;i,j) \, \text{ and } \, 
		\beta(M;i,j)  = \beta(M_1;i,j) \enspace .
	\]
	Further, if $i$ and $\ell$ are disconnected in $M$ then we have the independency
	\[
		\alpha(M;i,\ell) = \beta(M;i,\ell) = 1 \enspace .
	\]
\end{proposition}
\begin{proof}
	The number $b(M)$ of bases of a disconnected matroid $M=M_1\oplus M_2$ decomposes into the product $b(M_1)\cdot b(M_2)$.

	For elements $i,j$ in $M_1$ and $\ell$ $M_2$ we obtain the factorizations
	\begin{align*}
		b_i(M)    &= b_i(M_1)\cdot b(M_2),     & b_{ij}(M) &= b_{i,j}(M_1)\cdot b(M_2), & b^j_i(M) &= b^j_i(M_1)\cdot b(M_2),\\
		b^{ij}(M) &= b^{i,j}(M_1)\cdot b(M_2), & b_\ell(M) &= b(M_1)\cdot b_\ell(M_2),  & b_{i\ell}(M) &= b_i(M_1)\cdot b_\ell(M_2),\\
		b^\ell_i(M) &= b_i(M_1)\cdot b^\ell(M_2), & b^{i\ell}(M) &= b^i(M_1)\cdot b^\ell(M_2) \enspace.
	\end{align*}
	Substitution into the definitions provides the desired equations.
\end{proof}

Now let us proceed and include connected matroids in our considerations.
\begin{theorem}\label{thm:relation}
	Let $M$ be a matroid, $i$ and $j$ be neither loops nor parallel. Then one of the following four conditions holds
	\begin{align*}
		\alpha(M;i,j) >& \beta(M;i,j) > 1 \qquad\text{ or } & 
		1 = \alpha(M;i,j) =& \beta(M;i,j) \qquad \text{ or } \\
		0 < \alpha(M;i,j) <& \beta(M;i,j) < 1  \qquad\text{ or } &
		0 = \alpha(M;i,j) =& \beta(M;i,j) \enspace .
	\end{align*}
	In particular, for a  matroid with positive correlated elements the $\alpha$-ratio is an upper bound for its correlation constant.
\end{theorem}
\begin{proof} 
	The correlation constant of $M$ in terms of $b^{ij}$, $b_i^j$, $b_j^i$ and $b_{ij}$ is
	\begin{align} \label{eq:beta2alpha}
		\beta(M;i,j) & = \frac{(b^{ij}+b^j_i+b^i_j+b_{ij})\cdot b_{ij}}{(b^j_i+b_{ij})\cdot(b_j^i+b_{ij})} = \frac{b^{ij}\cdot b_{ij} + (b^j_i+b^i_j+b_{ij})\cdot b_{ij} }{b^j_i\cdot b^i_j+(b^j_i+b^i_j+b_{ij})\cdot b_{ij} }\enspace .
	\end{align}
	Clearly $b_{ij}=0$ implies the equality $\alpha(M;i,j)=\beta(M;i,j)=0$.
	The expression $ \frac{(b^j_i+b^i_j+b_{ij})\cdot b_{ij}}{b_i^j\cdot b_j^i}$ is positive whenever $b_{ij}$ does not vanish and the above Equation~\eqref{eq:beta2alpha} is equivalent to
	\begin{align*}
		\frac{(b^j_i+b^i_j+b_{ij})\cdot b_{ij}}{b_i^j\cdot b_j^i} \cdot \left( 1-\beta(M;i,j) \right) = \beta(M;i,j)-\alpha(M;i,j) \enspace .
	\end{align*}
	This confirms that exactly one of the four claimed cases applies.
\end{proof}

We will now take a look at a large class of matroids.
A matroid~$S$ is called \emph{sparse paving} if and only if every $r$-subset of $[n]$ is either a basis or a circuit of $S$.  
It is conjectured that almost all matroids are sparse paving; see \cite[Conjecture 1.6]{MayhewNewmanWelshWhittle:2011},\cite[Conjecture 15.5.10]{Oxley:2011}. 

\begin{proposition} The $\alpha$-ratio of a sparse paving matroid $S$ satisfies $\alpha(S) \leq \beta(S) \leq 1$.
\end{proposition}
\begin{proof}
	Let $S$ be a sparse paving $r$-matroid on $n$ elements.
	The deletion of an element and contraction of another leads to a sparse paving matroid of rank $r-1$ on $n-2$ elements. 
	Such a matroid has at least $\frac{n-r-1}{r-1} \tbinom{n-2}{r-2}$ bases if $2r\leq n$ and 
	at least $\frac{r-1}{n-r-1} \tbinom{n-2}{r}$ bases otherwise.
	This follows from \cite[Theorem 4.8]{MerinoETal:2012} and the fact that the class of sparse paving matroids is dually closed. These numbers give a lower bound on $b^j_i$ and $b^i_j$.
	Clearly, $b^{ij}$ is the number of bases of a $r$-matroid and $b_{ij}$ of a $(r-2)$-matroid on $n-2$ elements, hence these numbers are bounded from above by the corresponding binomial coefficients.
	Applying these bounds we get the following estimations.
	\begin{align*}
		\alpha(S) &\leq \left( \frac{r-1}{n-r-1} \right)^2 \frac{ \tbinom{n-2}{r}\tbinom{n-2}{r-2} }{\tbinom{n-2}{r-2}^2} 
		= \frac{ (r-1)\cdot (n-r)}{r\cdot(n-r-1)} \leq 1 \text{ if } 2r\leq n \text{ and }\\
		\alpha(S) &\leq \left( \frac{n-r-1}{r-1} \right)^2 \frac{ \tbinom{n-2}{r}\tbinom{n-2}{r-2} }{\tbinom{n-2}{r}^2} 
		= \frac{r \cdot (n-r-1) }{ (r-1)\cdot(n-r) } < 1 \text{ if } 2r > n .
	\end{align*}
	From $\alpha(S)\leq 1$ and Theorem~\ref{thm:relation} the claim follows.
\end{proof}


Now we want to construct a sequence of matroids $M_k$ with the property that $\beta(M_k)$ converges to $\alpha(M)$. 
Let $i$, $j$ be a valid pair of elements of the $r$-matroid $M$ on $n$ elements and $M_k$ the matroid that is obtained from $M$ by adding $k-1$ parallel copies of each element other than $i$ and $j$.   
The $r$-matroid $M_k$ consists of $k\cdot(n-2)+2$ elements, and the sequence of those matroids fulfills the desired property.

\begin{lemma}\label{lem:convergence}
	The sequence $\beta(M_k;i,j)$ converges monotonically to $\alpha(M;i,j) = \alpha(M_k;i,j)$.
\end{lemma}
\begin{proof}
	The numbers of bases of the deletions and contractions of the matroid $M_k$ satisfy the following equations due to the fact that we have $k$ choices to form a basis for every element in $M$ that is neither $i$ nor $j$.
	\begin{align*}
		b_{ij}(M_k) &= k^{r-2} \cdot b_{ij}(M), &
		b_{i}^j(M_k) &= k^{r-1} \cdot b_{i}^j(M)\\
		b_j^i(M_k) &= k^{r-1} \cdot b_j^i(M), &
		b^{ij}(M_k) &= k^{r} \cdot b^{ij}(M) \enspace .
	\end{align*}
	Hence $\alpha(M;i,j) = \alpha(M_k;i,j) $ and Equation~\eqref{eq:beta2alpha} turns into
	\[
		\beta(M_k;i,j) = \frac{ k^{2}\cdot (b^{ij}\cdot b_{ij}) + (k\cdot b^j_i+k\cdot b^i_j+b_{ij})\cdot b_{ij} }
		{k^{2}\cdot(b^j_i\cdot b^i_j)+(k\cdot b^j_i+k\cdot b^i_j+b_{ij})\cdot b_{ij} }
	\]
	 which converges clearly to $\alpha(M;i,j)$. The monotonicity can be read off from the numerator of the derivative with respect to $k$, which is $(b^{ij}\cdot b_{ij} - b_i^j\cdot b_j^i)\cdot ( k^2\cdot b_i^j+ k^2\cdot b_j^i+ 2k \cdot b_{ij})\cdot b_{ij}$.
\end{proof}

\section{Examples of matroids with positive correlations}
\noindent
Our aim in this section is to construct examples of (representable) matroids with a positive correlation.

Let $p$ be a prime number, $\FF_p$ the prime field of characteristic $p$ and $r\geq 2$ an integer.
	Consider the following vector configuration in $\FF_p^r$ given by the $2+p\cdot (r-1)$ vectors:
	\begin{align}\label{vectors}
	e_1,\; v=\sum_{\ell=2}^r e_\ell\; \text{ and } \;
	v_{k,\ell} = k\cdot e_1 + e_\ell \; \text{ for $1 < \ell \leq r$ and $0 \leq k < p$.}
	\end{align}
	Let $M_{r,p}$ denote the corresponding realizable $r$-matroid, with two special elements.
	The element $i$ that corresponds to the vector $e_1$, and the element $j$ that corresponds to $v$.

	To the best of the author's knowledge the following example is the only published example of a matroid with positive correlated elements.
\begin{example}\label{ex:SeymourWelsh} The matroid $M_{4,2}$ is the example given by Seymour and Welsh.
	Its correlation constant is $\beta(M_{4,2})= \frac{36}{35}$.
\end{example}

We now determine the numbers of bases $b_i^j$, $b_{ij}$, $b^{ij}$ and $b_j^i$ of all combinations of deletions and contractions of the two elements $i$ and $j$ in the matroid $M_{r,p}$.

The projection to the last $r-1$ coordinates corresponds to the contraction of $i$.
The obtained vector configuration consists of the all ones vector which is the projection of $v$, and $p$ copies of each of the $r-1$ standard vectors. Deleting the vector $v$ leads to $p$ choices of each standard vector.
Hence, $b_i^j = p^{r-1}$.

If we contract $j$, then in each basis exactly one of the $r-1$ standard vectors is not appearing and therefore $b_{ij} = (r-1)\cdot p^{r-2}$.

Note that for every index $\ell$ the three vectors $v_{k_1,\ell}$, $v_{k_2,\ell}$, $v_{k_3,\ell}$ are dependent.
Hence, an index maximally appears twice in a basis.
A consequence is that each basis that does not contain $e_1$ and $v$ consists of exactly one pair of vectors 
$v_{k_1,\ell}$, $v_{k_2,\ell}$ for an index $2 \leq \ell \leq r$ and values $0\leq k_1,k_2 < p$.
There are $r-1$ possibilities for the index $\ell$ and $\frac{p\cdot (p-1)}{2}$ choices for $k_1 \neq k_2$.
The vector $e_1$ lies in the span of $v_{k_1,\ell}$ and $v_{k_2,\ell}$ and with the arguments from before we get that we have $p^{r-2}$ choices for the other elements to form a basis.
We conclude that $b^{ij} = (r-1)\cdot\frac{(p-1)}{2}\cdot p^{r-1}$.

The last remaining case deals with the deletion of $i$ and contraction of $j$.
A set of vectors $\smallSetOf{v_{k_\ell,\ell}}{2\leq\ell\leq r}$ forms a basis with $v$ in the original vector configuration if and only if the sum $\sum_{\ell=2}^r k_\ell$ does not vanish. Clearly this sum depends on the characteristic.
In characteristic $p$ there are $p^{r-2}\cdot(p-1)$ of these bases.
There are further bases that contain a pair of vectors $v_{k_1,\ell}, v_{k_2,\ell}$ for an index $\ell$ and omit one of the other $r-2$ indices.  These are $(r-1)(r-2)\cdot \frac{p\cdot (p-1)}{2} p^{r-3}$ additional bases. In total this leads to $b_j^i= \frac{p^{r-2}\cdot(p-1)}{2} \cdot( 2 + (r-1)(r-2) )$.
We summarize our results.
\begin{lemma}
	Let $r\geq 4$ and $p$ be a prime number. The matroid $M_{r,p}$ has a positively correlated pair of elements. Its $\alpha$-ratio is
	$\alpha(M_{r,p}) = \frac{(r-1)^2}{2 + (r-1)\cdot(r-2)}$.
	This ratio is maximal $\frac{8}{7}$ for $r=5$.
\end{lemma}

With arguments as above, the configuration in \eqref{vectors} embedded in the rational vector space $\QQ^r$ yields $b_{ij} = (r-1) p^{r-2}$, $b^{ij}=(r-1)\cdot\frac{p-1}{2} p^{r-1}$, $b_i^j=p^{r-1}$ and $b_j^i=p^{r-1}-1+(r-1)(r-2)\frac{p-1}{2} p^{r-2}$, as the sum $\sum_{\ell=2}^r k_l=0$ if and only if all the non negative summands vanish. 
Hence the $\alpha$-ratio of this matroid converges to $\frac{(r-1)^2}{2+(r-1)(r-2)}$ as $p\to\infty$.

The next theorem combines the presented results with Huh's and Wang's bound.
\begin{theorem}\label{thm:final}
	The following inequalities hold for any class $\cC$ of $r$-matroids that is closed under parallel extentions and contains a matroid with a positive correlation.
	\[
		1 < \sup_{M\in\cC}\alpha(M) = \sup_{M\in\cC}\beta(M) \leq 2\cdot\frac{r-1}{r} \enspace .
	\]
	In particular, the correlation constant $\beta_\KK$ of the field $\KK$ satisfies $\frac{8}{7} \leq \beta_\KK \leq 2$.
\end{theorem}

\smallskip
\noindent
{\bf Acknowledgements.} 
Research by B. Schr\"oter is carried out in the framework of Matheon supported by Einstein Foundation Berlin.
The author thanks Institut Mittag-Leffler for the hospitality and support during the program ``Tropical Geometry, Amoebas and Polytopes''.
Moreover, he thanks Kristin Shaw for her useful comments and June Huh for introducing this problem to him.

\bibliographystyle{alpha}
\bibliography{References}

\begin{thebibliography}{MNRrInVF12}

\bibitem[AHK15]{AdiprasitoHuhKatz:2015}
Karim Adiprasito, June Huh, and Eric Katz.
\newblock Hodge theory for combinatorial geometries.
\newblock Preprint \arXiv{1511.02888}, 2015.

\bibitem[AHK17]{AdiprasitoHuhKatz:2017}
Karim Adiprasito, June Huh, and Eric Katz.
\newblock Hodge theory of matroids.
\newblock {\em Notices Amer. Math. Soc.}, 64(1):26--30, 2017.

\bibitem[Bak18]{Baker:2018}
Matthew Baker.
\newblock Hodge theory in combinatorics.
\newblock {\em Bull. Amer. Math. Soc. (N.S.)}, 55(1):57--80, 2018.

\bibitem[BSST40]{BrooksSmithStoneTutte:1940}
R.~L. Brooks, C.~A.~B. Smith, A.~H. Stone, and W.~T. Tutte.
\newblock The dissection of rectangles into squares.
\newblock {\em Duke Math. J.}, 7:312--340, 1940.

\bibitem[HW]{HuhWang}
June Huh and Botong Wang.
\newblock Mason's conjecture and the hodge--riemann relations for matroids.
\newblock In preparation.

\bibitem[HW17]{HuhWang:2017}
June Huh and Botong Wang.
\newblock Enumeration of points, lines, planes, etc.
\newblock {\em Acta Math.}, 218(2):297--317, 2017.

\bibitem[MNRrInVF12]{MerinoETal:2012}
Criel Merino, Steven~D. Noble, Marcelino Ram\'\i~rez Ib\'a\~nez, and Rafael
  Villarroel-Flores.
\newblock On the structure of the {$h$}-vector of a paving matroid.
\newblock {\em European J. Combin.}, 33(8):1787--1799, 2012.

\bibitem[MNWW11]{MayhewNewmanWelshWhittle:2011}
Dillon Mayhew, Mike Newman, Dominic Welsh, and Geoff Whittle.
\newblock On the asymptotic proportion of connected matroids.
\newblock {\em European J. Combin.}, 32(6):882--890, 2011.

\bibitem[Oxl11]{Oxley:2011}
James Oxley.
\newblock {\em Matroid theory}, volume~21 of {\em Oxford Graduate Texts in
  Mathematics}.
\newblock Oxford University Press, Oxford, second edition, 2011.

\bibitem[SW75]{SeymourWelsh:1975}
P.~D. Seymour and D.~J.~A. Welsh.
\newblock Combinatorial applications of an inequality from statistical
  mechanics.
\newblock {\em Math. Proc. Cambridge Philos. Soc.}, 77:485--495, 1975.

\bibitem[Tut74]{Tutte:1974}
W.~T. Tutte.
\newblock A problem on spanning trees.
\newblock {\em Quart. J. Math. Oxford Ser. (2)}, 25:253--255, 1974.

\bibitem[Wel71]{Welsh:1971}
D.~J.~A. Welsh.
\newblock Combinatorial problems in matroid theory.
\newblock In {\em Combinatorial {M}athematics and its {A}pplications ({P}roc.
  {C}onf., {O}xford, 1969)}, pages 291--306. Academic Press, London, 1971.

\bibitem[Whi86]{White:1986}
Neil White, editor.
\newblock {\em Theory of matroids}, volume~26 of {\em Encyclopedia of
  Mathematics and its Applications}.
\newblock Cambridge University Press, Cambridge, 1986.

\end{thebibliography}

\end{document}